\documentclass[12pt,reqno]{amsart}
\usepackage{comment}
\usepackage{amsmath,amsfonts,amsthm}
\usepackage{tikz}
\newtheorem{theorem}{Theorem}[section]
\newtheorem{prop}[theorem]{Proposition}
\newtheorem{coro}[theorem]{Corollary}

\usepackage[a4paper,
top=24mm,bottom=24mm,left=26mm,right=26mm
]{geometry}

\usepackage[numbers,sort]{natbib}
\usepackage[
bookmarks=true,
bookmarksdepth=subsection,
bookmarksnumbered=true,
breaklinks=true,
colorlinks=true,
linkcolor=blue,
citecolor=magenta
]{hyperref}

\usepackage[]{graphicx}
\usepackage{booktabs}
\usepackage{tikz}
\usetikzlibrary{cd}
\usetikzlibrary{arrows.meta}
\usetikzlibrary{matrix,arrows,decorations.pathmorphing}
\tikzcdset{arrow style=tikz,
  diagrams={>=Straight Barb}
}
\usetikzlibrary{positioning}
\usetikzlibrary{decorations.markings}
\tikzset{->-/.style={decoration={
      markings,
      mark=at position #1 with {\arrow{>}}}, postaction={decorate}}}
\tikzset{->>-/.style={decoration={
      markings,
      mark=at position #1 with {\arrow{>>}}}, postaction={decorate}}}

\usepackage{mathtools}
\usepackage{makecell}
  
\usepackage{mathrsfs}
\usepackage[sb]{libertine}
\usepackage[varqu,varl,scale=0.9]{zi4}
\usepackage[libertine,bigdelims,vvarbb]{newtxmath}
\usepackage[cal=boondoxo]{mathalfa} 
\numberwithin{equation}{section}
\newcommand{\I}{\mathrm{i}}
\newcommand{\E}{\mathrm{e}}

\DeclareMathOperator{\sh}{sh}
\DeclareMathOperator{\ch}{ch}
\DeclareMathOperator{\End}{End}

\DeclareMathOperator{\Sk}{Sk}

\DeclareMathDelimiter{\Norm}{\mathord}{largesymbols}{"3E}{largesymbols}{"3E}

\allowdisplaybreaks
\begin{document}

\baselineskip 15pt
\parskip 7pt
\sloppy


\title[]{Askey--Wilson Polynomial and
  Genus-Two Skein Module}


\author[]
{Kazuhiro Hikami}

\address{Faculty of Mathematics,
  Kyushu University,
  Fukuoka 819-0395, Japan.}

\email{
  \texttt{khikami@gmail.com}
}



\date{July 2, 2025.}

\begin{abstract}
  Based on  the $q$-difference operators for the
  genus-two skein algebra,
  we show the correspondence between the
  reduced
  Askey--Wilson polynomials and
  the skein module of the genus-2 handlebody.
\end{abstract}





\maketitle
\section{Introduction}

The skein algebra  is a fundamental tool in quantum topology.
As the Chern--Simons topological quantum field theory,
it has been used extensively in constructing quantum invariants of
knots and 3-manifolds such as the Jones polynomial and the
Witten--Reshetikhin--Turaev invariant (see,
\emph{e.g.},~\cite{KaufLins94Book,Licko97Book}).
Recently
it receives
renewed
interests in views from diverse topics in mathematics and physics,
such as
character variety,
Coulomb branch,
cluster algebra,
Teichm{\"u}ller theory,
and
topological phases of  matter.

One of the recent results on the skein algebra
is
a relationship with
the double affine Hecke algebra
(DAHA), which
was originally developed in studies of the symmetric  polynomials
associated to root systems
(see, \emph{e.g.},~\cite{Chered05Book,Macdonald03book}).
Crucial observations are 
that
the automorphisms of  DAHA of type-A and type-$C^\vee C_1$
can be regarded as
the Dehn twists 
on once-punctured torus~$\Sigma_{1,1}$
and 4-punctured sphere~$\Sigma_{0,4}$, respectively.
It was shown~\cite{IChered13a,IChered16a}
that
the super-polynomial of torus knots can be
recovered by use of actions of the Iwahori--Hecke operators of DAHA.

Studies of  DAHA for genus-two surface were initiated
in~\cite{ArthaShaki15a,ArthaShaki17a}.
Therein the $A_1$-type DAHA is generalized as  a $t$-deformed skein
algebra on  genus-two surface.
Later in~\cite{CookeSamue21a}, the genus-two skein module was studied
in detail, and
it was
proved  that the genus-two spherical DAHA
of~\cite{ArthaShaki15a,ArthaShaki17a} at~$t=q$ is isomorphic to the
genus-two skein algebra.
On the other hand,
we have
proposed~\cite{KHikami19a} 
different representations of $q$-difference operators for the genus-two skein algebra
by combining both type-$A_1$ and type-$C^\vee C_1$ DAHA.
We have
further constructed~\cite{KHikami24a}
the Iwahori--Hecke operators as a generalization of
type-$C^\vee C_1$, and discussed was the automorphisms associated
with the Dehn twists on $\Sigma_{2,0}$ as in the case of
$\Sigma_{1,1}$~\cite{IChered13a}.
Since then,
DAHA  for other surfaces has been studied~\cite{KHikami24b,KHikami25pre}.

The purpose of this article is to construct a map from
the genus-two skein
module~\cite{CookeSamue21a}
to
our  genus-two DAHA proposed in~\cite{KHikami19a,KHikami24a},
and to clarify the topological role of the (reduced) Askey--Wilson polynomial.
Namely we will prove that the $\theta$-link, which is a basis of the
genus-two skein module,
is mapped to the reduced Askey--Wilson polynomial.

This paper is organized as follows.
In section~\ref{sec:skein}, we recall the genus-two skein module
following~\cite{CookeSamue21a}.
We give the definitions of our
DAHA for the genus-two  skein algebra  in section~\ref{sec:DAHA20}.
In section~\ref{sec:AW},
reviewed are properties of the Askey--Wilson
polynomials.
We recall their $q$-shift operators for our later use.
In section~\ref{sec:homomorphism} we show the correspondence between
the reduced Askey--Wilson polynomials and the genus-two skein module.
The last section is devoted to concluding remarks.

Throughout this article, we use the standard Pochhammer symbol
(see, \emph{e.g.},~\cite{GaspRahm04}),
\begin{equation*}
  (x)_n = (x;q)_n=\prod_{i=1}^n (1-x \, q^{i-1}) ,
  \qquad
  (x_1,x_2, \dots, x_k)_n=(x_1)_n (x_2)_n \dots (x_k)_n .
\end{equation*}
The reflection operator~$\mathsf{s}$ and
the $q$-shift operators, $\eth$ and $\eth_b$,  are
defined respectively
by
\begin{gather*}
  \mathsf{s} f(x,x_0,x_1)
  =
  f(x^{-1},x_0,x_1),
  \\
  \eth f(x,x_0,x_1) = f(q^{\frac{1}{2}} x,x_0,x_1), \quad
  \eth_0 f(x,x_0,x_1)= f(x, q^{\frac{1}{2}}x_0, x_1) , \quad
  \eth_1 f (x,x_0,x_1)= f(x, x_0, q^{\frac{1}{2}} x_1) .
\end{gather*}
We also use a variant of the hyperbolic functions
defined by
\begin{equation*}
  \ch(x)=x+\frac{1}{x},
  \qquad
  \sh(x) = x- \frac{1}{x}.
\end{equation*}

\section{Genus-Two Skein Module}
\label{sec:skein}

The skein module $\Sk_A(M)$ of 3-manifold $M$ is a
$\mathbb{C}[A^{\pm 1}]$-module spanned by isotopy classes of framed
links in $M$
satisfying 
\begin{gather}
  \label{skein_relation_A}
  \vcenter{\hbox{
    \begin{tikzpicture}
      \draw [line width=1.2pt](0,1) --( 1,0) ;
      \draw[line width=10pt,white](0,0)--(1,1);
      \draw[line width=1.2pt](0,0)--(1,1);
    \end{tikzpicture}
  }}
  =
  A \,
  \vcenter{\hbox{
    \begin{tikzpicture}
      \draw [line width=1.2pt] (1,1) to[out=-135,in=135] (1,0);
      \draw [line width=1.2pt] (0,1) to[out=-45,in=45] (0,0);
    \end{tikzpicture}
  }}
  +A^{-1} \,
  \vcenter{\hbox{
    \begin{tikzpicture}
      \draw [line width=1.2pt] (1,1) to[out=-135,in=-45] (0,1);
      \draw [line width=1.2pt] (1,0) to[out=135,in=45] (0,0);
    \end{tikzpicture}
  }}
  ,
  \\[2mm]
  \vcenter{\hbox{
    \begin{tikzpicture}
      \draw [line width=1.2pt] (0,0) circle (0.5) ;
    \end{tikzpicture}
  }}
  = -A^2 - A^{-2} .
  \notag
\end{gather}
For a surface $\Sigma$,
we denote
$\Sk_A(\Sigma)$  as the skein module
of the thickened surface
$\Sigma\times[0,1]$.
A multiplication
$\mathbb{x} \, \mathbb{y}$
of links $\mathbb{x}$ and $\mathbb{y}$
means
that $\mathbb{x}$ is vertically above $\mathbb{y}$,
\begin{gather*}
  \mathbb{x} \, \mathbb{y}
  =
  \newcolumntype{C}{>{$}c<{$}}
  \begin{tabular}{|C|}
    \hline
    \hphantom{aa}\mathbb{x}\hphantom{aa}
    \\
    \hline
    \mathbb{y}
    \\
    \hline
  \end{tabular}
\end{gather*}

Generally
the skein algebra $\Sk_A(\Sigma)$ for a surface $\Sigma$ with
negative Euler characteristic
is generated by a finite set of
non-separating
simple closed curves~\cite{Santha24a}.
In the case of the skein algebra $\Sk_A(\Sigma_{2,0})$ on genus-two surface,
it is generated by the simple
closed curves~$\mathbb{k}_1, \dots, \mathbb{k}_5$
(see 
Fig.~\ref{fig:curves-20}) satisfying
several relations which  follow from the skein algebras on
sub-surfaces~$\Sigma_{0,4}$ and~$\Sigma_{1,1}$ of~$\Sigma_{2,0}$.
See~\cite{KHikami19a,KHikami24a,CookeSamue21a,ArthaShaki15a,Arthamo23a}
for explicit computations.

\begin{figure}[htbp]
  \centering
  \includegraphics[scale=.6]{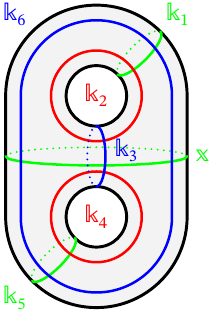}
  \qquad\qquad\qquad
  \includegraphics[scale=.6]{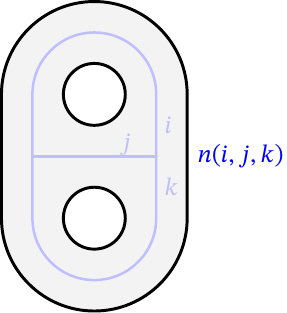}
  \caption{Simple closed curves on $\Sigma_{2,0}$  are given in the
    left.
    The right denotes
    $\theta$-link
    $n(i,j,k)$
    in the  genus-2 handlebody
    as a basis of
    the genus-two skein module.
  }

  \label{fig:curves-20}
\end{figure}

The $\theta$-link $n(i,j,k)$ in
Fig.~\ref{fig:curves-20} is a basis of the skein module of the genus-2
handlebody~\cite{Licko93b}.
Here
$(i,j,k)$ is an admissible triple satisfying
\begin{itemize}
\item $i, j, k \geq 0$,
\item $i+j+k$ is even,
\item $|i-j| \leq k \leq i+j$,
\end{itemize}
and
the trivalent  graph is defined by use of the Jones--Wenzl
idempotent.
See, \emph{e.g.},~\cite{KaufLins94Book,Licko97Book}, for its
definition and
roles
in constructing quantum invariants of knots and 3-manifolds.
The actions of the simple closed curves $\mathbb{k}_a$
on the
$\theta$-link~$n(i,j,k)$ were studied in detail~\cite{CookeSamue21a},
and they are  read under
the skein relations~\eqref{skein_relation_A} with $A=q^{-\frac{1}{4}}$
as
\begin{subequations}
  \label{k_and_theta-link}
  \begin{align}
    &
      \mathbb{k}_1  \, n(i,j,k)= -
      \ch(q^{\frac{1}{2}(i+1)})  \, n(i,j,k) ,
    \\
    &
      \mathbb{k}_3 \, n(i,j,k)= -
      \ch(q^{\frac{1}{2}(j+1)})  \, n(i,j,k) ,
    \\
    &
      \mathbb{k}_5  \, n(i,j,k)= -
      \ch(q^{\frac{1}{2}(k+1)})  \, n(i,j,k) ,
    \\
    &
      \mathbb{k}_2  \, n(i,j,k)
      =
      \sum_{a,b=\pm 1}
      D_{a,b}(i,j,k) \, n(i+a, j+b, k) ,
    \\
    &
      \mathbb{k}_4  \, n(i,j,k)
      =
      \sum_{a,b=\pm 1}
      D_{a,b}(j,k,i) \, n(i, j+a, k+b) ,
    \\
    &
      \mathbb{k}_6  \, n(i,j,k)
      =
      \sum_{a,b=\pm 1}
      D_{a,b}(i,k,j) \, n(i+a, j, k+b) ,
  \end{align}
\end{subequations}
where
\begin{subequations}
  \begin{align}
    &
      D_{1,1}(i,j,k)=1 ,
    \\
    &
      D_{1,-1}(i,j,k)=
      -q^{\frac{1}{2}(i+j-k+1)}
      \frac{
      \bigl(
      1-q^{\frac{1}{2}(-i+j+k)}
      \bigr)^2
      }{
      \left(1-q^{j} \right)
      \left(1-q^{j+1} \right)
      } ,
    \\
    &
      D_{-1,1}(i,j,k)=
      -q^{\frac{1}{2}(i+j-k+1)}
      \frac{
      \bigl(1-q^{\frac{1}{2}(i-j+k)} \bigr)^2
      }{
      \left( 1-q^{i} \right)
      \left( 1-q^{i+1} \right)
      } ,
    \\
    &
      D_{-1,-1}(i,j,k)=
      \frac{
      \bigl( 1-q^{\frac{1}{2}(i+j+k+2)} \bigr)^2
      \bigl( 1-q^{\frac{1}{2}(i+j-k)} \bigr)^2
      }{
      \left( 1-q^{i} \right) \left( 1-q^{i+1} \right)
      \left( 1-q^{j}\right) \left( 1-q^{j+1}\right)
      } .
  \end{align}
\end{subequations}
In~\cite{CookeSamue21a}, given 
was a map from~\eqref{k_and_theta-link} to the Arthamonov--Shakirov
$q$-difference operators~\cite{ArthaShaki15a,ArthaShaki17a}, which may be
regarded as a genus-two generalization of the spherical $A_1$ DAHA
although their Iwahori--Hecke structures are unknown.


\section{DAHA for $\Sigma_{2,0}$}
\label{sec:DAHA20}

Motivated by that the
$A_1$-type and the
$C^\vee C_1$-type spherical DAHAs are
isomorphic  to the skein algebra on once-punctured torus and
4-punctured sphere~\cite{Oblom04a,BerestSamuel16a,MortoSamue21a}
respectively,
we constructed 
the map
\begin{equation*}
  \mathcal{A}:
  \Sk_{A=q^{-\frac{1}{4}}}(\Sigma_{2,0}) 
  \to \End \mathbb{C}(q^{\frac{1}{4}},x_0,x_1)[x+x^{-1}]
\end{equation*}
in our previous papers~\cite{KHikami19a,KHikami24a}.
Explicitly
the simple closed curves in Fig.~\ref{fig:curves-20}
are mapped to
\begin{subequations}
  \label{define_A-k}
  \begin{align}
    &
      \mathcal{A}(\mathbb{k}_1)
      =\left. \ch(\I \, \mathsf{T}_0) \right|_{\text{sym}}
      = \ch(x_0) ,
    \\
    \label{T1T0_t-star}
    &
      \begin{aligned}[t]
        \mathcal{A}(\mathbb{k}_3)
        & =\left. \ch(\mathsf{T}_1\mathsf{T}_0) \right|_{\text{sym}}
        \\
        & =
          \sum_{\epsilon=\pm 1}
          \omega(x^\epsilon)
          \left\{
          -x^{-\epsilon}
          \left(x_0+ \frac{q^{\frac{1}{2}} x^\epsilon}{x_0} \right)
          \left(x_1+ \frac{q^{\frac{1}{2}} x^\epsilon}{x_1} \right)
          \eth^{2\epsilon}
          + q^{\frac{1}{2}} \ch(x_0) \ch(x_1)
          \right\} ,
      \end{aligned}
    \\
    &
      \mathcal{A}(\mathbb{k}_5)
      =\left. \ch(\I \, q^{-\frac{1}{2}}\mathsf{T}_1)
      \right|_{\text{sym}}
      = \ch(x_1) ,
    \\
    \label{A-k2_G}
    &
      \mathcal{A}(\mathbb{k}_2)
      =\left. \ch(\I \, \mathsf{U}_0) \right|_{\text{sym}}
      = \I  \, q^{-\frac{1}{4}} G_0(x_0;x) ,
    \\
    &
      \mathcal{A}(\mathbb{k}_4)
      =\left. \ch(\I \, q^{-\frac{1}{2}}\mathsf{U}_1)
      \right|_{\text{sym}}
      = \I \,  q^{-\frac{1}{4}} G_0(x_1;x) ,
    \\
    &
      \label{A-k6_U1U0}
      \begin{aligned}[t]
        \mathcal{A}(\mathbb{k}_6)
        & =\left. \ch( \mathsf{U}_1 \mathsf{U}_0) \right|_{\text{sym}}
        \\
        & =
          \sum_{\epsilon=\pm 1}
          \omega(x^\epsilon)
          \left\{
          K_0(x_0;x^\epsilon)  \, K_0(x_1;x^\epsilon) \eth^{2\epsilon}
          -G_0(x_0;x) \, G_0(x_1;x)
          \right\} .
      \end{aligned}
  \end{align}
\end{subequations}
Here ``sym'' denotes the action on the symmetric Laurent  polynomials
of~$x$,
and we mean
\begin{equation}
  \omega(x)= \frac{
    x(1+q^{\frac{1}{2}} x)}{
    q^{\frac{1}{2}} (1-x^2)(1-q^{\frac{1}{2}} x)
  } .
\end{equation}
By use of
the $q$-difference operators $K_n$ and $G_n$ defined by
\begin{subequations}
  \begin{align}
    \label{define_K_n}
    K_n(x_b; x)
    &= -\frac{x_b^{-n}}{1-x_b^2} \, \eth_b
    +
    \frac{x_b^n(q^{\frac{1}{2}} x + x_b^2) (q^{\frac{3}{2}} x+x_b^2)}{
    q x (1-x_b^2)} \,
    \eth_b^{-1} ,
    \\
    \label{define_G_n}
    G_n(x_b;x)
    & = - \frac{x_b^{-n}}{1-x_b^2}  \, \eth_b
      +
      \frac{x_b^n (q^{\frac{1}{2}}x + x_b^2) (q^{\frac{1}{2}}+x x_b^2)
      }{
      q^{\frac{1}{2}} x (1-x_b^2)
      } \,
      \eth_b^{-1} ,
  \end{align}
\end{subequations}
the Iwahori--Hecke operators~$\mathsf{T}_b$ and~$\mathsf{U}_b$ are written as~\cite{KHikami24a}
\begin{subequations}
  \label{define_U-T}
  \begin{align}
    \label{eq:3}
    \mathsf{T}_0
    &  \mapsto
      \I\frac{x}{q^{\frac{1}{2}}-x}
      \left(
      -\frac{q^{\frac{1}{2}} + x x_0^2}{x x_0} \,
      \mathsf{s} \,  \eth^2 + \ch(x_0)
      \right) ,
    \\
    \mathsf{T}_1
    & \mapsto
      \I
      \left(
      \frac{1+q^{\frac{1}{2}}x}{
      q^{\frac{1}{2}} (1-x^2)}
      \frac{q^{\frac{1}{2}} x + x_1^2}{x_1}
      \left(\mathsf{s} -1 \right)
      - q^{\frac{1}{2}}x_1^{-1}
      \right) ,
    \\
    \mathsf{U}_0
    & \mapsto
      \frac{q^{-\frac{1}{4}}x}{q^{\frac{1}{2}} - x}
      \left\{
      K_0(x_0;x^{-1}) \, \mathsf{s} \,  \eth^2 - G_0(x_0;x)
      \right\} ,
    \\
    \mathsf{U}_1
    & \mapsto
      -\frac{x(1+q^{\frac{1}{2}} x)}{
      q^{\frac{1}{4}} (1-x^2)
      }
      K_0(x_1;x) \left(\mathsf{s} -1 \right)
      +\frac{q^{\frac{1}{4}}}{1-q^{\frac{1}{2}}x}
      \left(
      G_0(x_1;x) - q^{\frac{1}{2}}x K_0(x_1;x)
      \right) ,
  \end{align}
\end{subequations}
which
satisfy the  Iwahori--Hecke relations,
\begin{subequations}
  \begin{gather}
      \left( \mathsf{T}_0 + \I \,  x_0\right)
      \left(\mathsf{T_0}+ \I \, x_0^{-1} \right)=0 ,      
    \\
      \left( \mathsf{T}_1 + \I  \, q^{-\frac{1}{2}}x_1 \right)
      \left( \mathsf{T}_1 + \I  \, q^{\frac{1}{2}}x_1^{-1} \right)
      =0,
    \\
      \left( \mathsf{U}_0 +
      \frac{q^{-\frac{1}{4}}x}{q^{\frac{1}{2}}-x}
      \left( G_0(x_0 ; x) - K_0(x_0; \tfrac{1}{x}) \right)
      \right)
      \left(
      \mathsf{U}_0 +
      \frac{q^{-\frac{1}{4}}}{q^{\frac{1}{2}}-x}
      \left(
      x  \, G_0(x_0;x) - q^{\frac{1}{2}} K_0(x_0; \tfrac{x}{q})
      \right)
      \right)=0 ,
    \\
      \left(
      \mathsf{U}_1 -
      \frac{q^{\frac{1}{4}}}{1-q^{\frac{1}{2}}x}
      \left( G_0(x_1 ; x) - q^{\frac{1}{2}}x \, K_0(x_1; x)
      \right)
  \right)
  \left(
  \mathsf{U}_1
  -\frac{q^{\frac{3}{4}}}{q^{\frac{3}{2}}-x}
      \left( G_0(x_1; \tfrac{x}{q}) - K_0(x_1; \tfrac{q}{x})
  \right)
  \right) =0 .
\end{gather}
\end{subequations}
It should be noted that
the operators $\mathsf{T}_b$ are the reductions of
the 
Iwahori--Hecke operators for  the  Askey--Wilson polynomial~\cite{NoumiStokm00a}.
See~\cite{KHikami24b} for an application of $\mathsf{U}_0$
to the skein algebra on
twice-punctured torus.

We remark that $x$ denotes the monodromy of the separating curve $\mathbb{x}$ in
Fig.~\ref{fig:curves-20}, and that
\begin{equation}
  \label{eq:1}
  \mathcal{A}(\mathbb{x})=
  \ch(x) .
\end{equation}

\section{Askey--Wilson Polynomial and $q$-Shift Operator}
\label{sec:AW}
The Askey--Wilson polynomials with 4 parameters
$\mathbf{t}=(t_0, t_1, t_2, t_3)$
are  defined as the basic  $q$-hypergeometric series~\cite{AsWi85}
(see also~\cite{GaspRahm04})
\begin{equation}
  \label{define_P_AW}
  \begin{aligned}[b]
    P_n(x; q, \mathbf{t})
    & =
      P_n^{(a,b,c,d)}
      =
      P_n^{(a,b,c,d)}(x)
    \\
    &  =
      \frac{
      \left( a b, a c, a d \right)_n
      }{
      a^n \left( a b c d q^{n-1} \right)_n
      }
      {}_4\phi_3
      \left[
      \begin{matrix}
        q^{-n}, q^{n-1} a b c d, a x, a x^{-1}
        \\
        a b , a c , a d
      \end{matrix}
      ; q, q
      \right] ,
  \end{aligned}
\end{equation}
where
\begin{equation*}
  a=\frac{1}{t_1 t_3}, \quad
  b=-\frac{t_3}{t_1}, \quad
  c= \frac{q^{\frac{1}{2}}}{t_0 t_2}, \quad
  d=-\frac{q^{\frac{1}{2}}t_2}{t_0} .
\end{equation*}
Here
the polynomials are symmetric in~$\{a, b, c, d\}$, and
we have normalized the symmetric Laurent polynomials so that
$P_n(x;q,\mathbf{t})=\ch(x^n)+\dots$,
where $\dots$ denotes a linear sum of $\ch(x^k)$ for
$0\leq k<n$.

We recall the $q$-difference operators introduced in~\cite{KalnMill89a};
\begin{subequations}
  \label{shift_Kalnins}
  \begin{align}
    \label{eq:16}
    & \widehat{\mathcal{m}}^{(a,b,c,d)}
      = \frac{1}{x-x^{-1}}
      \left(
      -\frac{1}{x} \left( 1 - a  q^{-\frac{1}{2}} x \right)
      \left( 1 - b q^{-\frac{1}{2}} x \right) \eth
      + x \left( 1-a q^{-\frac{1}{2}} x^{-1} \right)
      \left( 1- b q^{-\frac{1}{2}} x^{-1}  \right) \eth^{-1}
      \right) ,
    \\
    &\widehat{\mathcal{l}}^{(a,b,c,d)}
      =
      \frac{1}{x-x^{-1}} \left( \eth - \eth^{- 1 }
      \right) ,
    \\
    &
      \begin{multlined}[b][.84\textwidth]
        \widehat{\mathcal{l}}^{* (a q^{\frac{1}{2}} , b q^{\frac{1}{2}},
          c q^{\frac{1}{2}} , d q^{\frac{1}{2}})}
        =
        \frac{q^{-\frac{1}{2}}}{x-x^{-1}}
        \Bigl(
        \frac{
          (1- a x) (1- b x) ( 1- c x) ( 1-d x)}{x^2} \, 
        \eth
        \\
        -
        x^2  (1 - a x^{-1}) (1-b x^{-1}) (1-c x^{-1}) ( 1 - d x^{-1}) \,
        \eth^{-1}
        \Bigr) .
      \end{multlined}
  \end{align}
\end{subequations}
These act as parameter $q$-shift operators
on the Askey--Wilson polynomial~\eqref{define_P_AW} as
\begin{subequations}
  \label{action_shift_Kalnins}
  \begin{align}
    \label{eq:14}
    & \widehat{\mathcal{m}}^{(a,b,c,d)} P_n^{(a,b,c,d)}
      =
      q^{-\frac{n}{2}} (1 - a b q^{n-1} ) \,
      P_n^{(a q^{-\frac{1}{2}}, b q^{-\frac{1}{2}} , c q^{\frac{1}{2}},
      d q^{\frac{1}{2}})} ,
    \\
    & \widehat{\mathcal{l}}^{(a,b,c,d)} P_n^{(a,b,c,d)}
      = -q^{-\frac{n}{2}} ( 1-q^n) \,
      P_{n-1}^{(a q^{\frac{1}{2}}, b q^{\frac{1}{2}},
      c q^{\frac{1}{2}}, d q^{\frac{1}{2}})} ,
    \\
    & \widehat{\mathcal{l}}^{* (a q^{\frac{1}{2}} , b q^{\frac{1}{2}},
      c q^{\frac{1}{2}} , d q^{\frac{1}{2}})}
      P_n^{ (a q^{\frac{1}{2}} , b q^{\frac{1}{2}},
      c q^{\frac{1}{2}} , d q^{\frac{1}{2}})}
      =
      -q^{-\frac{n+1}{2}} (1 - a b c d q^n ) \,
      P_{n+1}^{(a,b,c,d)} .
  \end{align}
\end{subequations}
These $q$-shift operators were employed to simplify the proof of the orthogonality 
of~$P_n^{(a,b,c,d)}$~\cite{KalnMill89a}.

For our purpose of the genus-two skein algebra~$\Sk_A(\Sigma_{2,0})$,
we pay attention to the case~\cite{KHikami19a}
\begin{equation}
  \label{t-star}
  \mathbf{t}_\star=
  \left( \I \, x_0, \I \, q^{-\frac{1}{2}} x_1,
    \I \,  x_0, \I \,  x_1 \right) .
\end{equation}
Hereafter for simplicity
we denote the Askey--Wilson polynomial~\eqref{define_P_AW} at
$\mathbf{t}_\star$
as
\begin{align}
  \label{our_AW_from_symmetric}
  P_n
  & =
    P_n(x; x_0, x_1)
    =
    P_n^{(
    -q^{\frac{1}{2}}, - q^{\frac{1}{2}},
    - q^{\frac{1}{2}}/x_0^2,
    - q^{\frac{1}{2}}/x_1^2
    )}
  \\
  \nonumber
  &
    =
    P_n(x;q, \mathbf{t}_\star)
    =
    (-1)^n q^{-\frac{n}{2}}
    \frac{
    \left( q, \frac{q}{x_0^2}, \frac{q}{x_1^2} \right)_n}{
    \left( \frac{q^{n+1}}{x_0^2 x_1^2} \right)_n
    } 
    \sum_{k=0}^n
    q^k \,
    \frac{
    \left(q^{-n}, \frac{q^{n+1}}{x_0^2 x_1^2} \right)_k}{
    \left( q,q,\frac{q}{x_0^2} , \frac{q}{x_1^2}\right)_k
    }  \, g_k(x) ,
\end{align}
where  $g_k(x)$ is a basis of the symmetric Laurent polynomial
\begin{equation}
  \label{our_symmetric_base}
  g_k(x)=(-q^{\frac{1}{2}}x , -q^{\frac{1}{2}}x^{-1})_k .
\end{equation}
The expansion in terms of~$g_k(x)$
reminds the Habiro expansion of the colored Jones
polynomial~\cite{KHabiro06b}.

The reduced Askey--Wilson polynomial $P_n$
is characterized as an eigenpolynomial of the reduced
Askey--Wilson $q$-difference
operator
\begin{equation}
  \label{eigen_AW}
  \ch\left( \mathsf{T}_1 \mathsf{T}_0 \right) 
  P_n
  =
  - \ch \left( q^{-n-\frac{1}{2}} x_0 x_1 \right) 
  P_n ,
\end{equation}
where $\mathsf{T}_b$ are the Iwahori--Hecke operators
in~\eqref{define_U-T}.
See~\eqref{T1T0_t-star} for the explicit form of the reduced Askey--Wilson operator.
As an orthogonal polynomial,
the three-term recurrence relation for~$P_n$
is read as~\cite{AsWi85}~\cite[(7.5.3)]{GaspRahm04}
\begin{equation}
  \label{3-term-AW}
  \ch(x) \,
  P_n
  =
  P_{n+1} 
  +
  \beta_n(x_0,x_1) \,
  P_n 
  + \gamma_n(x_0,x_1) \,
  P_{n-1} , 
\end{equation}
where
\begin{align}
  \label{define_Beta_n}
  &
    \begin{multlined}[b][.84\textwidth]
      \beta_n(x_0,x_1)
      =
    q^{-\frac{1}{2}}
    \left\{
      \frac{
        \left(1-q^{n+1}\right)
        \left( 1- \frac{q^{n+1}}{x_0^2 x_1^2} \right)
        \left( 1- \frac{q^{n+1}}{x_0^2 } \right)
        \left( 1- \frac{q^{n+1}}{x_1^2} \right)
      }{
        \left(1- \frac{q^{2n+1}}{x_0^2 x_1^2} \right)
        \left(1- \frac{q^{2n+2}}{x_0^2 x_1^2} \right)
      }
      -1\right\}
    \\
    +
    q^{\frac{1}{2}}
    \left\{
      \frac{
        \left(1-q^{n}\right)
        \left( 1- \frac{q^{n}}{x_0^2 x_1^2} \right)
        \left( 1- \frac{q^{n}}{x_0^2 } \right)
        \left( 1- \frac{q^{n}}{x_1^2} \right)
      }{
        \left(1- \frac{q^{2 n}}{x_0^2 x_1^2} \right)
        \left(1- \frac{q^{2 n+1}}{x_0^2 x_1^2} \right)
      }
      -1\right\} ,
  \end{multlined}
  \\
  \label{define_Gamma_n}
    &
      \gamma_n(x_0,x_1) =
      \frac{
      (1-q^n)^2 \left(1-\frac{q^n}{x_0^2 x_1^2} \right)^2
      \left(1-\frac{q^n}{x_0^2 } \right)^2
      \left(1-\frac{q^n}{x_1^2} \right)^2
      }{
      \left(1-\frac{q^{2 n-1}}{x_0^2 x_1^2} \right)
      \left(1-\frac{q^{2 n}}{x_0^2 x_1^2} \right)^2
      \left(1-\frac{q^{2 n+1}}{x_0^2 x_1^2} \right)
      } .
\end{align}
We should note that,
at $\mathbf{t}_\star$~\eqref{t-star},
the connection formula~\cite{AsWi85}~\cite[(7.6.2)]{GaspRahm04} is simplified to be
\begin{equation}
  \label{connection-AW}
  P_n(x;q^{\frac{1}{2}}x_0 , x_1)
  =
  P_n 
  +
  \lambda_n(x_0,x_1) \,
  P_{n-1} , 
\end{equation}
where
\begin{equation}
  \label{define_lambda}
  \lambda_n(x_0,x_1)
  =
  \frac{q^{-\frac{1}{2}} (1-q^n)^2
    \left(1-\frac{q^n}{x_1^2}\right)^2}{
    x_0^2
    \left( 1- \frac{q^{2 n-1}}{x_0^2 x_1^2} \right)
    \left(1- \frac{q^{2 n}}{x_0^2 x_1^2} \right)
  } .
\end{equation}
We also have the similar identity by transposing $x_0$ and $x_1$.
We find that
the functions in~\eqref{define_Beta_n},~\eqref{define_Gamma_n}, and~\eqref{define_lambda},
satisfy
\begin{equation}
  \label{Beta_and_lambda_Gamma}
  \beta_n(x_0,x_1) +
  \ch(  x_0^2q^{\frac{1}{2}}) 
  =
  \lambda_{n+1}(x_0,x_1)+ \frac{\gamma_n(x_0,x_1)}{\lambda_n(x_0,x_1)} .
\end{equation}

In the following, we study the actions of $\mathcal{A}(\mathbb{k}_a)$~\eqref{define_A-k}
on the
reduced Askey--Wilson
polynomial defined by
\begin{equation}
  \label{define_bar-P}
  \widebar{P}_n(x;x_0,x_1)
  =
  \nu_n(x_0,x_1)  \, P_n(x; x_0,x_1) ,
\end{equation}
with the scale function $\nu_n(x_0,x_1)$ to be determined.

\begin{prop}
  \label{prop:action}
  The action of the $q$-difference operators
  $\mathcal{A}(\mathbb{k}_a)$~\eqref{define_A-k} on the
  reduced Askey--Wilson polynomial~\eqref{define_bar-P} are given as follows;
  \begin{subequations}
    \label{A-action_on_P-bar}
  \begin{align}
    \label{A-action-1}
    & \mathcal{A}(\mathbb{k}_1) \,
      \widebar{P}_n
      =
      \ch(x_0) \, \widebar{P}_n,
    \\
    & \mathcal{A}(\mathbb{k}_3) \,
      \widebar{P}_n
      =
      - \ch\left(
      q^{-n-\frac{1}{2}}x_0 x_1
      \right) 
      \widebar{P}_n ,
    \\
    \label{A-action-5}
    & \mathcal{A}(\mathbb{k}_5) \,
      \widebar{P}_n
      =
      \ch(x_1) \, \widebar{P}_n ,
    \\
    \label{k2_AW}
    &
      \begin{multlined}[b][.84\textwidth]
        \mathcal{A}(\mathbb{k}_2) \,
        \widebar{P}_n
        =
        \I \frac{q^{-\frac{1}{4}}}{1-x_0^2}
        \Biggl\{
                x_0^2 \frac{
          \nu_n(q^{-\frac{1}{2}}x_0, x_1)}{
          \nu_{n+1}(x_0,x_1)}
        \widebar{P}_{n+1}
        -
        \frac{\nu_n(q^{\frac{1}{2}}x_0,x_1)}{
          \nu_{n-1}(x_0,x_1)} \,
        \lambda_n(x_0,x_1) \,
        \widebar{P}_{n-1}
        \\
        +
        \left(
          x_0^2 \, \frac{\gamma_n(q^{-\frac{1}{2}}x_0,x_1)}{
            \lambda_n (q^{-\frac{1}{2}}x_0, x_1)}
          \frac{
            \nu_n(q^{-\frac{1}{2}}x_0, x_1)}{
            \nu_n(x_0, x_1)}
          -
          \frac{
            \nu_n(q^{\frac{1}{2}}x_0, x_1)}{
            \nu_n(x_0, x_1)}
        \right)
        \widebar{P}_n
        \Biggr\} ,
      \end{multlined}
    \\
    \label{k4_AW}
    &
      \begin{multlined}[b][.84\textwidth]
        \mathcal{A}(\mathbb{k}_4) \,
        \widebar{P}_n
        =
        \I \frac{q^{-\frac{1}{4}}}{1-x_1^2}
        \Biggl\{
        x_1^2 \, \frac{
          \nu_n(x_0, q^{-\frac{1}{2}}x_1)}{
          \nu_{n+1}(x_0,x_1)}
        \widebar{P}_{n+1}
        -
        \frac{\nu_n(x_0, q^{\frac{1}{2}}x_1)}{
          \nu_{n-1}(x_0,x_1)} \,
        \lambda_n(x_1,x_0) \,
        \widebar{P}_{n-1}
        \\
        +
        \left(
          x_1^2 \, \frac{\gamma_n(x_0, q^{-\frac{1}{2}}x_1)}{
            \lambda_n ( q^{-\frac{1}{2}}x_1, x_0)}
          \frac{
            \nu_n(x_0, q^{-\frac{1}{2}}x_1)}{
            \nu_n(x_0, x_1)}
          -
          \frac{
            \nu_n(x_0, q^{\frac{1}{2}}x_1)}{
            \nu_n(x_0, x_1)}
        \right)
        \widebar{P}_n
        \Biggr\} ,
      \end{multlined}
    \\ 
    \label{A-k6_P6-bar}
    &
      \begin{multlined}[b][.84\textwidth]
        \mathcal{A}(\mathbb{k}_6) \,
        \widebar{P}_n
        =
        \frac{1}{(1-x_0^2)(1-x_1^2)}
        \Biggl\{
                \frac{\nu_n(q^{-\frac{1}{2}}x_0, q^{-\frac{1}{2}}
          x_1)}{
          \nu_{n+1}(x_0,x_1)}
        q^{-n-\frac{5}{2}}
        \left( x_0^2 x_1^2  - q^{n+2} \right)^2
        \widebar{P}_{n+1}
        \\
        -q^{-n-\frac{3}{2}}
        \left(
          \frac{\nu_n(q^{\frac{1}{2}}x_0,
            q^{-\frac{1}{2}}x_1)}{
            \nu_n(x_0,x_1)} 
          \left(x_1^2-q^{n+1}\right)^2
          +
          \frac{\nu_n(q^{-\frac{1}{2}}x_0,
            q^{\frac{1}{2}}x_1)}{
            \nu_n(x_0,x_1)} 
          \left(x_0^2- q^{n+1}\right)^2
        \right) \, \widebar{P}_n
        \\
        +
        \frac{
          \nu_n(q^{\frac{1}{2}}x_0, q^{\frac{1}{2}}x_1)}{
          \nu_{n-1}(x_0,x_1)}
        q^{-n-\frac{1}{2}} (1-q^n)^2 \,
        \widebar{P}_{n-1}
        \Biggr\} .
      \end{multlined}
  \end{align}
\end{subequations}
\end{prop}

\begin{proof}
  The cases for $\mathbb{k}_1$ and $\mathbb{k}_5$  are trivial.
  The action of $\mathcal{A}(\mathbb{k}_3)$
  is from the definition of the reduced Askey--Wilson operator~\eqref{eigen_AW}.

  For $\mathbb{k}_2$~\eqref{A-k2_G},
  the definition~\eqref{define_G_n}
  gives
  \begin{align*}
    & \mathcal{A}(\mathbb{k}_2) \, \widebar{P}_n
      =
      \I \, \frac{q^{-\frac{1}{4}}}{1-x_0^2}
      \left\{
      - \widebar{P}_n(x;q^{\frac{1}{2}}x_0, x_1)
      +
      x_0^2 \left(
      \ch\left(q^{-\frac{1}{2}}x_0^2 \right)
      +\ch(x)
      \right) \, \widebar{P}_n(x; q^{-\frac{1}{2}}x_0 , x_1)
      \right\}
    \\
    & =
      \begin{multlined}[t][.84\textwidth]
      \I \, \frac{q^{-\frac{1}{4}}}{1-x_0^2}
      \Biggl\{
      -   \frac{\nu_n(q^{\frac{1}{2}}x_0, x_1)
      }{\nu_n(x_0,x_1)} \, \widebar{P}_n
      - \frac{\nu_n(q^{\frac{1}{2}}x_0,x_1)
      }{ \nu_{n-1}(x_0,x_1)}  \, \lambda_n(x_0,x_1) \,
      \widebar{P}_{n-1}
      \\
      +x_0^2 \frac{\nu_n(q^{-\frac{1}{2}}x_0,x_1)}{
        \nu_{n+1}(q^{-\frac{1}{2}}x_0,x_1)} \,
      \widebar{P}_{n+1}(x;q^{-\frac{1}{2}}x_0,x_1)
      +x_0^2 \left( \ch ( q^{-\frac{1}{2}}x_0^2 )
        +\beta_n( q^{-\frac{1}{2}} x_0, x_1 )
        \right) \, \widebar{P}_n(x;q^{-\frac{1}{2}}x_0,x_1)
        \\
        +x_0^2 \frac{
          \nu_n(q^{-\frac{1}{2}}x_0,x_1)}{
          \nu_{n-1}(q^{-\frac{1}{2}}x_0,x_1)} \,
        \gamma_n(q^{-\frac{1}{2}}x_0, x_1) \,
        \widebar{P}_{n-1}(x;q^{-\frac{1}{2}}x_0,x_1)
      \Biggr\} ,
    \end{multlined}
  \end{align*}
  where, in the second equality,
  we have used~\eqref{3-term-AW} and~\eqref{connection-AW}.
  Using~\eqref{Beta_and_lambda_Gamma}
  and~\eqref{connection-AW},
  we obtain~\eqref{k2_AW}.
  The action~\eqref{k4_AW} of $\mathcal{A}(\mathbb{k}_4)$ is given similarly
  by
  transposing~$x_0$ and~$x_1$.

  For $\mathbb{k}_6$~\eqref{A-k6_U1U0},
  we may rewrite the  operator~$\mathcal{A}(\mathbb{k}_6)$ as
  \begin{equation}
    \label{A-k6_with_q-d}
    \mathcal{A}(\mathbb{k}_6)
    =
    \frac{1}{(1-x_0^2)( 1-x_1^2)}
    \sum_{a, b =\pm 1} (-1)^{\frac{a+b}{2}+1}
    \widehat{\mathcal{d}}_{a,b} \eth_0^a \eth_1^b ,
  \end{equation}
  where $\widehat{\mathcal{d}}_{a,b}$ are the $q$-difference operators
  for $x$ defined by
  \begin{subequations}
    \label{difference_in_k6}
    \begin{align}
      \widehat{\mathcal{d}}_{1,1}
      &=
        \sum_{\epsilon=\pm 1}
        \omega(x^\epsilon )
        \left( \eth^{2\epsilon} -1
        \right) ,
      \\
      \widehat{\mathcal{d}}_{1,-1}
      &    =
        \sum_{\epsilon=\pm 1}
        \omega(x^\epsilon)
        \left(
        b(x_1,x^\epsilon) \eth^{2\epsilon} - c(x_1,x)
        \right) ,
      \\
      \widehat{\mathcal{d}}_{-1,1}
      &    =
        \sum_{\epsilon=\pm 1}
        \omega(x^\epsilon)
        \left(
        b(x_0,x^\epsilon) \eth^{2\epsilon} - c(x_0,x)
        \right) ,
      \\
      \widehat{\mathcal{d}}_{-1,-1}
      & =
        \sum_{\epsilon=\pm 1}
        \omega(x^\epsilon)
        \left(
        b(x_0,x^\epsilon) b(x_1,x^\epsilon) \eth^{2\epsilon}
        -c(x_0,x) c(x_1,x)
        \right) .
    \end{align}
  \end{subequations}
  Here we have used
  \begin{equation*}
    b(x_b,x)=
    \frac{
      \bigl(q^{\frac{1}{2}} x + x_b^2 \bigr)
      \bigl(q^{\frac{3}{2}} x + x_b^2 \bigr)
    }{q \, x},
    \qquad
    c(x_b,x)=
    \frac{
      \bigl(q^{\frac{1}{2}} x + x_b^2 \bigr)
      \bigl(q^{\frac{1}{2}}  + x \, x_b^2 \bigr)
    }{q^{\frac{1}{2}} \, x}.
  \end{equation*}
  Recalling   the $q$-shift operators~\eqref{shift_Kalnins} for the
  Askey--Wilson polynomials,
  we can check by straightforward computations
  that 
  the operators~$\widehat{\mathcal{d}}_{a,b}$~\eqref{difference_in_k6}
  are factorized as
  \begin{subequations}
    \begin{align}
      \label{eq:17}
      \widehat{\mathcal{d}}_{1,1}
      & =
        -q^{-\frac{1}{2}} \,
        \widehat{\mathcal{l}}^{(-1,-1,c q^{-\frac{1}{2}} , d q^{-\frac{1}{2}})}
        \,
        \widehat{\mathcal{m}}^{(-q^{\frac{1}{2}} , - q^{\frac{1}{2}},
        c q^{-1}, d q^{-1})} ,
      \\
      \widehat{\mathcal{d}}_{1,-1}
      & =
        q^{-\frac{3}{2}} \, x_1^4 \,
        \widehat{\mathcal{m}}^{(-q, -q/x_1^2, - 1/x_0^2, -1)}
        \widehat{\mathcal{m}}^{(-q^{\frac{1}{2}} , -
        q^{\frac{3}{2}}/x_1^2,
        - q^{-\frac{1}{2}}/x_0^2, - q^{\frac{1}{2}})} ,
      \\
      \widehat{\mathcal{d}}_{-1,1}
      & =
        q^{-\frac{3}{2}}  \, x_0^4 \,
        \widehat{\mathcal{m}}^{(-q, -q/x_0^2, - 1/x_1^2, -1)}
        \widehat{\mathcal{m}}^{(-q^{\frac{1}{2}} ,
        - q^{\frac{3}{2}}/x_0^2,
        - q^{-\frac{1}{2}}/x_1^2,
        - q^{\frac{1}{2}})} ,
      \\
      \widehat{\mathcal{d}}_{-1,-1}
      & =
        -q^{-2} x_0^4 \,  x_1^4 \, 
        \widehat{\mathcal{l}}^{* ( -q/x_0^2,
        -q/x_1^2, -q, -q)}
        \widehat{\mathcal{m}}^{(-q^{\frac{3}{2}}/x_0^2 ,
        -  q^{\frac{3}{2}}/x_1^2, - q^{\frac{1}{2}}, -   q^{\frac{1}{2}})} .
    \end{align}
  \end{subequations}
  Here $c$ and $d$ in $\widehat{\mathcal{d}}_{1,1}$ are arbitrary.
  Then we  obtain
  the actions~\eqref{action_shift_Kalnins} on the Askey--Wilson polynomials
  as
  \begin{subequations}
    \label{k6_action_part}
    \begin{gather}
      \label{k6_supplement-1}
      \widehat{\mathcal{d}}_{1,1} P_n ( x; q^{\frac{1}{2}} x_0, q^{\frac{1}{2}} x_1)
      =
      q^{-n-\frac{1}{2}}
      (1-q^n)^2 \,
      P_{n-1} ,
      \\
      \label{k6_supplement-2}
      \widehat{\mathcal{d}}_{1,-1}
      P_n \left( x; q^{\frac{1}{2}} x_0, q^{-\frac{1}{2}}x_1 \right)
      =
      q^{-n-\frac{3}{2}}    \left(x_1^2 - q^{n+1} \right)^2
      P_n ,
      \\
      \label{k6_supplement-2a}
      \widehat{\mathcal{d}}_{-1,1}
      P_n \left( x; q^{-\frac{1}{2}} x_0, q^{\frac{1}{2}}x_1 \right)
      =
      q^{-n-\frac{3}{2}}    \left( x_0^2 - q^{n+1} \right)^2 \,
      P_n ,
      \\
      \label{k6_supplement-3}
      \widehat{\mathcal{d}}_{-1,-1}
      P_n\left(
        x; q^{-\frac{1}{2}}x_0 , q^{-\frac{1}{2}}x_1
      \right)
      =
      q^{-n-\frac{5}{2}}   
      \left( x_0^2 x_1^2 - q^{n+2} \right)^2
      P_{n+1} .
    \end{gather}
  \end{subequations}
  Combining the results~\eqref{k6_action_part} with~\eqref{A-k6_with_q-d},
  we get~\eqref{A-k6_P6-bar}.
\end{proof}

\section{Map from the Genus-Two Skein Module}
\label{sec:homomorphism}

In this section
we fix the scale function $\nu_n(x_0,x_1)$ in~\eqref{define_bar-P}.
Prop.~\ref{prop:action} reduces to the following.
\begin{coro}
When we set the scale function as
\begin{align}
  \label{define_nu}
  \nu_n(x_0,x_1)
  & =
    (-1)^n q^{-\frac{1}{2}n(n+1)}
  \left(  \tfrac{q^{n+1}}{x_0^2 x_1^2}\right)_n
  \left(  \tfrac{q^{n+1}}{x_0^2 x_1^2}\right)_{n+1}
  \prod_{b=0,1}
  \frac{
    x_b^{\frac{1}{2}} \E^{ \pi \I \frac{\log x_b}{\log q}}
    }{
    \left(\frac{1}{x_b^2}\right)_{n+1}
    \left(\frac{q}{x_b^2}\right)_{n}
    (x_b^2)_\infty
  } ,
\end{align}
then the
actions of~\eqref{A-action_on_P-bar} are  read   as
\begin{subequations}
  \label{k_special-nu}
  \begin{align}
    \label{k_special-nu-1}
    &
      \mathcal{A}(\mathbb{k}_1) \widebar{P}_n
      =\ch(x_0) \, \widebar{P}_n ,
    \\
    &
      \mathcal{A}(\mathbb{k}_3) \widebar{P}_n
      = - \ch\left(q^{-n-\frac{1}{2}}x_0 x_1 \right) \, \widebar{P}_n ,
    \\
    \label{k_special-nu-5}
    &
      \mathcal{A}(\mathbb{k}_5) \widebar{P}_n
      =\ch(x_1) \, \widebar{P}_n ,
    \\
    \label{k2_special-nu}
    &
      \begin{multlined}[b][.84\textwidth]
        \mathcal{A}(\mathbb{k}_2) \widebar{P}_n
        =
        - 
        \frac{
          q^{-n-\frac{3}{2}}x_0^2
          \left( 1- x_1^2 q^{-n-1}\right)^2
        }{
          \left(1- x_0^2 x_1^2 q^{-2n-3} \right)
          \left(1- x_0^2 x_1^2 q^{-2n-2} \right)
        } \,
        \widebar{P}_{n+1}
        - 
        \frac{
          q^{-n+\frac{1}{2}}
          \left( 1-q^n \right)^2
          x_0^2}{
          (1-x_0^2 ) ( 1 - q \, x_0^2)
        } \,
        \widebar{P}_{n-1}
        \\
        +
        \left( 1 +
          \frac{
            ( 1- x_0^2 q^{-n})^2
            (1- x_0^2 x_1^2 q^{-n})^2
          }{
            (1-x_0^2) ( 1- q \,x_0^2)
            (1-x_0^2 x_1^2 q^{-2n-1})
            (1-x_0^2 x_1^2 q^{-2n})
          }
          \right) \,
          \widebar{P}_n ,
      \end{multlined}
    \\
    &
      \begin{multlined}[b][.84\textwidth]
        \mathcal{A}(\mathbb{k}_4) \widebar{P}_n
        =
        - 
        \frac{
          q^{-n-\frac{3}{2}}x_1^2
          \left( 1- x_0^2 q^{-n-1}\right)^2
        }{
          \left(1- x_0^2 x_1^2 q^{-2n-3} \right)
          \left(1- x_0^2 x_1^2 q^{-2n-2} \right)
        } \,
        \widebar{P}_{n+1}
        - 
        \frac{
          q^{-n+\frac{1}{2}}
          \left( 1-q^n \right)^2
          x_1^2}{
          (1-x_1^2 ) ( 1 - q \, x_1^2)
        } \,
        \widebar{P}_{n-1}
        \\
        +
        \left( 1 +
          \frac{
            ( 1- x_1^2 q^{-n})^2
            (1- x_0^2 x_1^2 q^{-n})^2
          }{
            (1-x_1^2) ( 1- q \,x_1^2)
            (1-x_0^2 x_1^2 q^{-2n-1})
            (1-x_0^2 x_1^2 q^{-2n})
          }
          \right) \,
          \widebar{P}_n ,
      \end{multlined}
    \\
    &
      \begin{multlined}[b][.84\textwidth]
        \mathcal{A}(\mathbb{k}_6) \widebar{P}_n
        =
        \widebar{P}_{n+1}
        + \frac{(1-q^n)^2 (1-x_0^2 x_1^2 q^{-n+1})^2}{
          (1-x_0^2) ( 1- q \, x_0^2)
          (1-x_1^2) (1-q \, x_1^2)
        } \, \widebar{P}_{n-1}
        \\
        -
        \left(
          \frac{
            q^{n+\frac{1}{2}} (1-x_0^2 q^{-n})^2}{
            (1-x_0^2) (1- q \, x_0^2)
          }
          +
          \frac{
            q^{n+\frac{1}{2}} (1-x_1^2 q^{-n})^2}{
            (1-x_1^2) (1- q  \, x_1^2)
          }
        \right) \,
        \widebar{P}_n .
      \end{multlined}
  \end{align}
\end{subequations}
\end{coro}

It is now straightforward to see the correspondence with the genus-two
skein module~\eqref{k_and_theta-link}.
To construct explicitly
the map from the skein module,
we introduce
$\left| i,j,k \right\rangle$
for the admissible triple $(i,j,k)$, on which~$x_b$ and the $q$-difference operator~$\eth_b$ act as
\begin{equation}
  \label{eq:8}
  \begin{gathered}
    x_0  \left| i,j,k \right\rangle
    = -q^{\frac{i+1}{2}} \left| i,j,k \right\rangle,
    \qquad \qquad
    x_1  \left| i,j,k \right\rangle
    = -q^{\frac{k+1}{2}} \left| i,j,k \right\rangle,
    \\
    \eth_0^a \eth_1^b \left| i,j,k \right\rangle
    =
    \left|i-a, j, k-b \right\rangle
    .
  \end{gathered}
\end{equation}
We then consider a composition of maps
\begin{equation}
  \label{map_for_module}
  n(i,j,k) \mapsto \left| i,j,k \right\rangle
  \mapsto
  \widebar{P}_{\frac{i-j+k}{2}}  ,
\end{equation}
associated to
$  \mathbb{k}_a \mapsto \mathcal{A}(\mathbb{k}_a)$
in~\eqref{define_A-k}.
The coefficients~$D_{a,b}(i,j,k)$ in~\eqref{k_and_theta-link}
can be rewritten as the actions of~$x_b$
on~$\left| i,j,k\right\rangle$.
It is easy to see that
the first three actions of~\eqref{k_and_theta-link}
map to~\eqref{k_special-nu-1}--\eqref{k_special-nu-5}.
We can check that
the remaining actions of~\eqref{k_and_theta-link} also map to
those in~\eqref{k_special-nu} by~\eqref{map_for_module}.
To conclude,
we have shown that the reduced Askey--Wilson
polynomials~$\widebar{P}_n$~\eqref{define_bar-P}
with~\eqref{define_nu} play a role of
the $\theta$-link~$n(i,j,k)$ in
the genus-two skein module~\eqref{k_and_theta-link}.

\section{Concluding Remarks}
We have shown the correspondence between the genus-two skein algebra
in~\cite{CookeSamue21a} and the family of the $q$-difference operators
including the reduced Askey--Wilson operators~\cite{KHikami19a}.
Given is
the topological
picture of the reduced Askey--Wilson polynomial
$\widebar{P}_n$~\eqref{define_bar-P}.
As a generalization of the Askey--Wilson polynomial,
it will be promising to study 
the 
roles of the $C^\vee C_n$-type DAHA and the
Macdonald--Koornwinder polynomials in the genus-two HOMFLY
skein algebra.

\section*{Acknowledgments}
The work of KH is supported in part by
JSPS KAKENHI Grant Numbers
JP23K22388.


\end{document}